\pdfoutput=1
\RequirePackage{ifpdf}
\ifpdf 
\documentclass[pdftex]{sigma}
\else
\documentclass{sigma}
\fi

\numberwithin{equation}{section}

\newtheorem{Theorem}{Theorem}[section]
\newtheorem{Corollary}[Theorem]{Corollary}
\newtheorem{Lemma}[Theorem]{Lemma}
\newtheorem{Proposition}[Theorem]{Proposition}
 { \theoremstyle{definition}
\newtheorem{Definition}[Theorem]{Definition}
\newtheorem{Example}[Theorem]{Example}
\newtheorem{Remark}[Theorem]{Remark} }

\newcommand{\half}{\ensuremath{\frac{1}{2}}}

\newcommand{\N}{{\mathbb N}}
\newcommand{\R}{{\mathbb R}}

\newcommand{\Rd}{{{\mathbb R}^{d}}}
\newcommand{\Sdmone}{{{\mathbb S}^{d-1}}}

\newcommand{\HypFunTO}[4]{{
        {_2F_1}\left[ \left.
        \begin{array}{@{}c@{}}
        #1 , #2\\
        #3 \\
        \end{array} \right| #4 \right] }}

 \def\CD{{\mathcal D}}

 \def\CI{{\mathcal I}}

 \def\NN{{\mathbb N}}

 \def\cov{\operatorname{Cov}}
 \def\Lip{\operatorname{Lip}}
 
 \DeclareMathOperator{\opD}{{{\mathcal D}}} 
 \DeclareMathOperator{\opI}{{{\mathcal I}}} 

\newcommand*{\Cdot}[1][2.0]{
  \mathpalette{\CdotAux{#1}}\cdot%
}
\newdimen\CdotAxis
\newcommand*{\CdotAux}[3]{%
  {%
    \settoheight\CdotAxis{$#2\vcenter{}$}%
    \sbox0{%
      \raisebox\CdotAxis{%
        \scalebox{#1}{%
          \raisebox{-\CdotAxis}{%
            $\mathsurround=0pt #2#3$%
          }%
        }%
      }%
    }%
    \dp0=0pt %
    \sbox2{$#2\bullet$}%
    \ifdim\ht2<\ht0 %
      \ht0=\ht2 %
    \fi
    \sbox2{$\mathsurround=0pt #2#3$}%
    \hbox to \wd2{\hss\usebox{0}\hss}%
  }%
}

\begin{document}

\allowdisplaybreaks

\renewcommand{\thefootnote}{$\star$}

\newcommand{\arXivNumber}{1601.07743}

\renewcommand{\PaperNumber}{043}

\FirstPageHeading

\ShortArticleName{One-Step Recurrences for Stationary Random Fields on the Sphere}

\ArticleName{One-Step Recurrences for Stationary Random Fields\\ on the Sphere\footnote{This paper is a~contribution to the Special Issue
on Orthogonal Polynomials, Special Functions and Applications.
The full collection is available at \href{http://www.emis.de/journals/SIGMA/OPSFA2015.html}{http://www.emis.de/journals/SIGMA/OPSFA2015.html}}}

\Author{R.K.~BEATSON~$^\dag$ and W.~ZU~CASTELL~$^{\ddag\S}$}

\AuthorNameForHeading{R.K.~Beatson and W.~zu~Castell}

\Address{$^\dag$~School of Mathematics and Statistics, University of Canterbury,\\
\hphantom{$^\dag$}~Private Bag 4800, Christchurch, New Zealand}
\EmailD{\href{mailto:r.beatson@math.canterbury.ac.nz}{r.beatson@math.canterbury.ac.nz}}
\URLaddressD{\url{http://www.math.canterbury.ac.nz/~r.beatson}}

\Address{$^\ddag$~Scientific Computing Research Unit, Helmholtz Zentrum M\"{u}nchen,\\
\hphantom{$^\ddag$}~Ingolst\"{a}dter Landstra{\ss}e~1, 85764 Neuherberg, Germany}
\EmailD{\href{mailto:castell@helmholtz-muenchen.de}{castell@helmholtz-muenchen.de}}
\URLaddressD{\url{http://www.helmholtz-muenchen.de/asc}}

\Address{$^\S$~Department of Mathematics, Technische Universit\"{a}t M\"{u}nchen, Germany}

\ArticleDates{Received January 28, 2016, in f\/inal form April 15, 2016; Published online April 28, 2016}

\Abstract{Recurrences for positive def\/inite functions in terms of the space dimension have been used in several f\/ields of applications. Such recurrences typically relate to properties of the system of special functions characterizing the geometry of the underlying space. In the case of the sphere ${\mathbb S}^{d-1} \subset {\mathbb R}^d$ the (strict) positive def\/initeness
of the zonal function $f(\cos \theta)$ is determined by the signs of the coef\/f\/icients in the expansion of~$f$
in terms of the Gegenbauer polynomials $\{C^\lambda_n\}$, with $\lambda=(d-2)/2$.
Recent results show that classical dif\/ferentiation and integration applied to~$f$ have positive
def\/initeness preserving properties in this context. However, in these results the space dimension changes in steps of two.
This paper develops operators for zonal functions on the sphere which preserve (strict) positive def\/initeness
while moving up and down in the ladder of dimensions by steps of one.
These fractional operators are constructed to act appropriately
on the Gegenbauer polyno\-mials~$\{C^\lambda_n\}$.}


\Classification{42A82; 33C45; 42C10; 62M30}

\renewcommand{\thefootnote}{\arabic{footnote}}
\setcounter{footnote}{0}

\section{Introduction}
\label{sec:introduction}
This paper develops operators for zonal functions on the sphere which preserve (strict) positive def\/initenesss
while moving up and down in the ladder of dimensions by steps of one. The operators provide tools for
forming families of (strictly) positive def\/inite zonal functions. Such (strictly) positive def\/inite zonal functions
 can be used as covariance models for
estimating regionalized variables and also for interpolation on spheres.

Within a deterministic context, zonal positive def\/inite functions on the sphere have been used for interpolation or approximation of scattered data (see \cite{Fa98,Fr98} and the references therein). The standard ansatz in this setting is a linear combination of spherical translates of a f\/ixed \emph{$($zonal$)$ basis function}. While the present paper
could well have been stated within the context of approximation on the sphere, we rather
chose to provide a probabilistic framework, which is to some extent is equivalent, i.e., the
theory of regionalized variables.

Regionalized variables on spherical domains can nicely be modeled using
random f\/ields on spheres \cite{Bi73, Le59}. Such a random f\/ield is given through a~set of random variables, $Z(x)$ say, where $x\in\Sdmone$. Assuming the f\/ield
to be Gaussian, i.e., for every $n\in{\mathbb N}$, $(Z(x_1),\dots,Z(x_n))^T$ has a~multivariate
Gaussian distribution for any choice of $x_1,\dots,x_n\in\Sdmone$,
the distribution can be characterized by its f\/irst two moments.

Assuming second order (weak) stationarity, the covariance for an
isotropic model is determined by a function
\begin{gather*}
\cov\big(Z(x),Z(y)\big) = f(\cos\theta),
\qquad x,y\in\Sdmone,
\end{gather*}
where $\theta=\theta(x,y)=\arccos\big(x^Ty\big)$ is the geodesic distance between the points
$x$ and $y$ on the sphere $\Sdmone$.

As a consequence of \emph{Kolmogorov's extension theorem}
(see \cite[Theorem~36.3]{Bi79}), the set of isotropic
Gaussian random f\/ields can be identif\/ied with the set of zonal positive def\/inite
functions on the sphere \cite{Ga67, Le59}. We note in passing that L\'evy named such processes \emph{Brownian motion}.

\begin{Definition}A continuous function $g\colon [0,\pi] \to \R$ is (zonal) positive def\/inite on the sphe\-re~$\Sdmone$ if for all~$n\in\N$ and all distinct point sets
$\{x_1,\ldots, x_n\}$ on the sphere, the inequality
\begin{gather*}
\sum_{i,j=1}^n c_i c_j g(\theta(x_i,x_j)) \geq 0
\end{gather*}
holds true for all $c_1,\dots, c_n\in\R$. The function is (zonal) strictly positive def\/inite on~$\Sdmone$ if the inequality holds in the strict sense for all $c_1,\dots , c_n\in\R$ not vanishing simultaneously.
\end{Definition}

Although the natural distance on the unit sphere is an angle in $[0,\pi]$, it is convenient for the purpose of this paper to consider functions in $x=\cos\theta\in[-1,1]$, instead.
Thus, by $\Lambda_{d-1}$ we will denote the cone of all functions $f\in C[-1,1]$ such that $f(\cos\, \cdot\,)$ is positive def\/inite on $\Sdmone$. $\Lambda^+_{d-1}$ will denote the subcone of all strictly positive def\/inite functions in~$\Lambda_{d-1}$.

Gaussian random f\/ields have been widely applied to statistically analyze spatial phenomena \cite{Ch99,Ma65,Sc12}. In particular, kriging allows prediction of
spatial variables from given samples at arbitrary locations. The key ingredient
for such an approach lies in determining a suitable model for the covariance
function of the spherical random f\/ield. Commonly, such a model can be inferred
from given data through f\/itting a parametric family of models (i.e., estimation of the
covariance).

Models for covariance functions have further been used for simulation of stationary random f\/ields. Matheron \cite{Ma73} suggested a method based on proper averaging of stationary random f\/ields on a lower dimensional space. In the Euclidean setting this so-called \emph{turning bands method} works as follows:

Given a stationary random f\/ield $Z_1$ on the real line with covariance function $C_1$ and a~randomly chosen direction $\xi\in\Sdmone$, $Z_\xi=Z_1(x^T\xi)$ def\/ines a stationary random f\/ield on~$\Rd$. Avera\-ging over all directions $\xi$ leads to a stationary f\/ield on $\Rd$ the covariance function of which, $C_d$~say, relates to~$C_1$ via the so-called \emph{turning bands operator}
\begin{gather*}
C_d(t) = \operatorname{const} \int_0^\infty \big(1-\tau^2\big)_+^{\frac{d-3}2} C_1(t\tau) d\tau,
\qquad t\in\R_+.
\end{gather*}
The turning bands operator represents one example out of a suite of operators, mapping radial positive def\/inite functions on $\Rd$ onto such functions on a higher or lower dimensional space. Wendland~\cite{We95}, Wu~\cite{Wu95} and
Gneiting~\cite{Gn02} used such operators to derive compactly supported functions of a given smoothness.
Recurrences for radial positive def\/inite functions in general have been studied by several authors \cite{Sc96, Ca02}. Due to
Schoenberg's characterization of radial positive def\/inite functions and
the fact that scale mixtures of such functions preserve positive def\/initeness,
recurrence operators can be derived from corresponding relations between
special functions. In the case of radial functions on~$\Rd$, the appropriate fundamental relation is \emph{Sonine's first integral} for Bessel functions of the f\/irst kind (see~\cite{Ca02}).

In a recent paper \cite{Be15} the authors applied similar operators to derive parametrized families of suitable locally supported covariance models for stationary random f\/ields on the sphere $\Sdmone$. These operators are based on properties of Gegenbauer polynomials, appearing in Schoenberg's characterization \cite{Sc42} of zonal positive def\/inite functions on the sphere.

\begin{Theorem} \label{thm:schoenberg} Let $\lambda =(d-2)/2$, and consider a continuous function $f$ on $[-1,1]$.
The function $f(\cos\,\cdot\,)$ is positive definite on $\Sdmone$, i.e., $f \in \Lambda_{d-1}$, if and only if
$f$ has an ultraspherical expansion
\begin{gather} \label{eq:expansion_in_schoenberg_thm}
f(x) \sim \sum_{n=0}^\infty a_n C^{\lambda}_n(x),
\qquad x\in[-1,1],
\end{gather}
in which all the coefficients $a_n$ are nonnegative,
and the series converges at the point $x=1$.
If this is the case, the series converges absolutely and uniformly on the whole interval.
\end{Theorem}

Chen, Menegatto and Sun \cite{Ch03} showed that a necessary and suf\/f\/icient condition for $f\circ \cos$ to be strictly positive def\/inite on $\Sdmone$, $d\geq 3$, is that, in addition to the conditions of Theorem~\ref{thm:schoenberg}, inf\/initely many of the Gegenbauer coef\/f\/icients~$a_n$ with odd index, and inf\/initely many of those with even index, are positive. In the case $d=2$ the criteria is necessary but not suf\/f\/icient for~$f\circ \cos$ to be strictly positive def\/inite. A characterization in this case has been
given by Menegatto, Oliveira \& Peron~\cite{Me06}, although the criterion is a little more involved (see also~\cite{Ba16} for further details on these issues).

In the same spirit as for the turning bands method, a zonal function def\/ined on a lower dimensional sphere ${\mathbb S}^{d-\kappa}$
can be lifted up to~$\Sdmone$ through averaging
over the set of all copies of~${\mathbb S}^{d-\kappa}$ contained in~$\Sdmone$.
In~\cite{Be15} it is shown that the analogues for the sphere of Matheron's \emph{mont\'{e}e} and \emph{descente} operators (see~\cite{Ma65}) for $\R^d$ are the operators
\begin{gather*}
(\opI f)(x) = \int_{-1}^x f(u) du, \qquad x\in [-1,1],
\end{gather*}
and
\begin{gather*}
(\opD f)(x) = f'(x), \qquad x \in [-1,1].
\end{gather*}
Paralleling the behaviour of Matheron's operators in the Euclidean case the operators move in the
ladder of dimensions by steps of two. Specif\/ically, the $\opI$ and $\opD$ operators map zonal positive def\/inite functions $f( \cos \cdot)$ on ${\mathbb S}^d$ onto ones on ${\mathbb S}^{d-2}$ and ${\mathbb S}^{d+2}$, respectively (see~\cite{Be15} for details). Therefore, the natural question arises, whether it would also be possible to proceed through steps by one within the ladder of dimensions. While in the Euclidean case this could be achieved using fractional dif\/ferentiation and integration (see \cite{Ca02}), the situation is more intriguing in the spherical setting. The reason lies in the fact that the characterizing special functions for the sphere are polynomials, which are not preserved through fractional integration. Thus, one has to work with combinations of fractional operators in order to guarantee that the operators are mapping into the space of polynomials.

In the present paper, we provide a suite of four operators which can be used to def\/ine a~\emph{clavier} (cf.~\cite{Ma65}) for the sphere. The main results are given in
Theorems~\ref{thm:dimension_hopping_4},~\ref{thm:dimension_hopping_5} and~\ref{thm:half_step_D_general_case}, below.
We start with introducing the appropriate
fractional operators in the following section and studying their action on ultraspherical expansions.
The action of the operators on Gegenbauer polynomials shown in the
last section is derived using properties of hypergeometric $_2F_1$-functions.

\section{Def\/inition of the half-step operators}

In the expansion (\ref{eq:expansion_in_schoenberg_thm}) the dimension $d$ appears in the parameter $\lambda=(d-2)/2$ of the Gegenbauer polynomials. This relation between $\lambda$ and $d$ will be f\/ixed throughout the paper.

From $\opD C_n^\lambda=2\lambda C_{n-1}^{\lambda+1}$ (cf.~\cite[10.9(23)]{Er53ii}) we see that classical dif\/ferentiation and its inverse, integration, alter the parameter $\lambda$ by an integer. This is why the operators
$\opI$ and $\opD$ traverse the ladder of dimensions in steps of two (see~\cite{Be15}). At the same time, $\opI$ and $\opD$ change the degree of polynomials by one. Therefore, in order to obtain a one-step operator in the dimension, we have to consider fractional integration and dif\/ferentiation, a fact which perfectly parallels the Euclidean setting (see~\cite{Po07, Ca02}).

We are now ready to def\/ine the half-step operators and discuss their action on positive def\/inite functions on~$\Sdmone$.

\begin{Definition}
For $f\in L^1[-1,1]$ and $\lambda\geq 0$, def\/ine
\begin{gather}
I^\lambda_+ f(x) = I^{\lambda,\frac 12}_+ f(x) =
(1+x)^{-\lambda+\half}
\int_{-1}^x (x-\tau)^{-\frac 12} (1+\tau)^\lambda f(\tau) d\tau, \label{eq:Iplus}\\
I^\lambda_{-} f(x) = I^{\lambda,\frac 12}_{-} f(x) =
(1-x)^{-\lambda+\frac 12}
\int_x^1 (\tau-x)^{-\frac 12} (1-\tau)^{\lambda } f(\tau) d\tau. \label{eq:Iminus}
\end{gather}
Using these, we further def\/ine
\begin{gather}
\CI^\lambda_+ = I^\lambda_{+}+I^\lambda_{-}\qquad \text{and} \qquad
\CI^\lambda_- = I^\lambda_{+}-I^\lambda_{-}. \label{eq:CI_plus_minus}
\end{gather}
\end{Definition}

Apart from the additional factor $(1\pm x)^{-\lambda+\frac 12}$ in front of the integral and the weight \mbox{$(1\pm\tau)^\lambda$}, the operators $I^\lambda_\pm$ are classical
\emph{Riemann--Liouville fractional integrals} of order~$\frac 12$
(cf.\ \cite[De\-f\/i\-ni\-tion~2.1]{Sa93}) on the interval $[-1,1]$.
To def\/ine inverse operators, we use the corresponding \emph{Riemann--Liouville fractional derivates}
 (cf.~\cite[Def\/inition~2.2]{Sa93}).

\begin{Definition}
Let $f$ be absolutely continuous on $[-1,1]$ and $\lambda\geq 0$. Then
\begin{gather*}
D^\lambda_+ f(x) = D^{\lambda,\frac 12}_+ f(x) =
(1+x) \frac{d}{dx} \left\{ (1+x)^{-\lambda} \int_{-1}^x (x-\tau)^{-\half} (1+\tau)^{\lambda-\half} f(\tau) d\tau \right\},\\
D^\lambda_{-} f(x) = D^{\lambda,\frac 12}_{-} f(x) =
(1-x) \frac{d}{dx} \left\{ (1-x)^{-\lambda} \int_{x}^1
(\tau-x)^{-\half} (1-\tau)^{\lambda-\half} f(\tau) d\tau \right\}.
\end{gather*}
Using these, we further def\/ine
\begin{gather*}
\CD^\lambda_+ = D^\lambda_{+}+D^\lambda_{-}\qquad \text{and} \qquad
\CD^\lambda_- = D^\lambda_{+}-D^\lambda_{-}.
\end{gather*}
\end{Definition}

The main results of this paper are the following two theorems giving precise statements of the dimension hopping
and positive def\/initeness preserving properties of
the operators $\CI^\lambda_{\pm}$ and $\CD^\lambda_{\pm}$. These are
 one-step analogues of Theorems 2.2 and~2.3 in~\cite{Be15}.
Since
in the light of Theorem~\ref{thm:schoenberg} the statements can be considered as statements concerning ultraspherical expansions without referring back to a sphere, we are considering $m=d-1$ to be a positive integer.

\begin{Theorem} \label{thm:dimension_hopping_4}
Let $m$ be a positive integer and $\lambda=(m-1)/2$.
\begin{itemize}\itemsep=0pt
\item[$(a)$]
 \begin{itemize}\itemsep=0pt
\item[$(i)$] Let $f \in \Lambda_{m+1}$, $m \geq 1$. Then $ \CI^\lambda_{\pm} f \in \Lambda_{m}$.
 \item[$(ii)$] Let $f \in \Lambda_{m+1}^+$, $m \geq 2$. Then $\CI^\lambda_{\pm} f \in \Lambda_{m}^+$.
 \end{itemize}
 \item[$(b)$] Let $m\geq 1$, $f \in \Lambda_{m+1}^+$ be nonnegative, and $f$ have Gegenbauer expansion,
 \begin{gather*}
f \sim \sum_{n=0}^\infty a_n C^{\lambda+\half}_n,
 \end{gather*}
with all coefficients, $\{ a_n\}_{n=0}^\infty$, positive.
Then $ \CI^\lambda_{+} f$ is also nonnegative, $ \CI^\lambda_+ f \in \Lambda_m^+$,
 and all the coefficients $b_n$ in the expansion
 \begin{gather*}
 \CI^\lambda_+ f \sim \sum_{n=0}^\infty b_n C^{\lambda}_n,
 \end{gather*}
 are positive.
 \end{itemize}
\end{Theorem}

\begin{proof}
The proofs for the statements are almost identical with those of the corresponding parts of Proposition~2.2 in~\cite{Be15}, provided that proper analogues for certain statements on Gegenbauer polynomials are given. We therefore restrict ourselves to pointing out where adaptations of the proof given in~\cite{Be15} are needed.

One of these details concerns
the boundedness of the operators $\CI^\lambda_{\pm}$ as operators from $C[-1,1]$ to
$C[-1,1]$. This follows from the def\/initions of $I^\lambda_{\pm}$ and $\CI^\lambda_{\pm}$ in equations~\eqref{eq:Iplus},~\eqref{eq:Iminus} and~\eqref{eq:CI_plus_minus}, combined with the beta integrals
 \begin{gather} \label{eq:beta_integral1}
 \int_{-1}^x (x-\tau)^{-\half} (1+\tau)^{\nu} d \tau=
 (1+x)^{\nu+\half} B\left(\tfrac{1}{2}, \nu+1\right) = (1+x)^{\nu+\half} \frac{\Gamma(\half)\Gamma(\nu+1)}{\Gamma(\nu + \frac 32)},
 \end{gather}
and
 \begin{gather*} 
 \int_{x}^1 (\tau-x)^{-\half} (1-\tau)^{\nu} d \tau=
 (1-x)^{\nu+\half} B\left(\tfrac{1}{2}, \nu+1\right) = (1-x)^{\nu+\half} \frac{\Gamma(\half)\Gamma(\nu+1)}{\Gamma(\nu + \frac 32)}.
\end{gather*}
Similarly, positivity of the operator $\CI^\lambda_+$ follows from the def\/initions~\eqref{eq:Iplus},~\eqref{eq:Iminus} and~\eqref{eq:CI_plus_minus}. The main ingredient thus remaining to be shown is the action of $\CI^\lambda_{\pm}$ on the Gegenbauer polyno\-mial~$C^{\lambda+\frac12}_n$. This part is given in Theorem~\ref{thm:action_on_polynomials}, below.
Note that in contrast to the operators studied in~\cite{Be15}, there is no need to deal with an extra constant in statements~(i) and~(ii). This follows from Theorem~\ref{thm:action_on_polynomials}, showing that the operators~$\CI^\lambda_{\pm}$
do not introduce an additional constant.
\end{proof}

\begin{Theorem} \label{thm:dimension_hopping_5}
Suppose that $f \in \Lambda_m$, $m\geq 1$, and let $\lambda =(m-1)/2$.
Then, if both functions $\CD^\lambda_{\pm} f \in C[-1,1]$, then $\CD^\lambda_{\pm}f \in \Lambda_{m+1}$.
If, in addition, $f\in \Lambda_{m}^+$,
then $\CD^\lambda_{\pm}\in \Lambda_{m+1}^+$.
\end{Theorem}

\begin{Remark} Since the operators def\/ined above can be seen as standard operators
of fractional integration/dif\/ferentiation, classical results from fractional
calculus can be applied. For example,
if $(1+\tau)^{-1/2}f(\tau)\in\Lip\alpha$ for some $\alpha>\frac 12$,
in particular, if $f\in\Lip\alpha$ and supp$f\subset(-1,1]$, then
by Theorem~19 in~\cite{Ha28}~$\CD_\pm^0 f$ exists and is continuous. An analogous
statement holds for general~$\lambda$.
\end{Remark}

The proof of Theorem~\ref{thm:dimension_hopping_5}
depends heavily on a multiplier relationships between the Gegenbauer coef\/f\/icients of~$f$ and those
of $\CD^\lambda_{\pm}f$. The details of these relationship, and the proof of Theorem~\ref{thm:dimension_hopping_5},
will be deferred to the next subsection.

Let us f\/inish the section with considering an example.

\begin{Example}
Consider the operator $I_+^\lambda$. In view of its def\/inition~(\ref{eq:Iplus})
this operator maps functions locally supported near one to functions locally
supported near one. Also, since
$I_+^\lambda=(\CI_+^\lambda+\CI_-^\lambda)/2$ this operator preserves
(strict) positive def\/initeness by Theorem~\ref{thm:dimension_hopping_4}.

Note that by a change of variables
\begin{gather*}
\int_{-1}^x (x-\tau)^{-\frac 12}(1+\tau)^\lambda f(\tau) d\tau
= (x+1)^{\lambda+\frac 12}
\int_0^1 (1-s)^{-\frac 12}s^\lambda f\big((x+1)s-1\big) ds.
\end{gather*}
Therefore, if $f$ were such that $f\big((x+1)s-1\big)=(1-ys)^{-a}$, the
integral becomes
\begin{gather}
\label{eq:Euler_integral}
\int_0^1 (1-s)^{-\frac 12}s^\lambda(1-ys)^{-a} ds
= \frac{\Gamma(\lambda+1)\Gamma\left(\frac 12\right)}
 {\Gamma\left(\lambda+\frac 32\right)}\,
	\HypFunTO{a}{\lambda+1}{\lambda+\frac 32}{y},
\end{gather}
being a special case of \emph{Euler's integral} for hypergeometric functions
(cf.\ \cite[\href{http://dlmf.nist.gov/15.6.E1}{(15.6.1)}]{DLMF}).
\end{Example}

Now consider the \emph{Cauchy family}
\begin{gather*}
\varphi_{\alpha,\beta}(r) = (1+r^\alpha)^{-\frac\beta\alpha},
\qquad 0<\alpha\leq 2, \qquad \beta>0,
\end{gather*}
which is strictly positive def\/inite on ${\mathbb R}^d$
for all $d\geq 1$ (cf.~\cite{Gn02}).
Choosing $\alpha=2$ and restricting the function $\varphi_{2,\beta}$ to
the sphere, we obtain (setting $\tau=\cos\theta$)
\begin{gather*}
\varphi_\beta(\tau) = \varphi_{2,\beta}\big(\sqrt{2-2\cos\theta}\big) =
(3-2\tau)^{-\frac\beta 2},
\qquad \beta>0.
\end{gather*}
Thus,
\begin{gather*}
\varphi_\beta\big( (x+1)s-1\big)
=
5^{-\frac\beta 2} \left( 1-\tfrac 25(x+1)s\right)^{-\frac\beta 2}.
\end{gather*}
Therefore, from~(\ref{eq:Euler_integral}) we have that
\begin{gather*}
I_+^\lambda \varphi_\beta(x) =
\frac{\sqrt\pi\Gamma(\lambda+1)}{5^{\frac\beta 2}\Gamma\left(\lambda+\frac 32\right)}
(x+1) \HypFunTO{\frac\beta 2}{\lambda+1}{\lambda+\frac 32}{\tfrac 25(x+1)}.
\end{gather*}
Since (see \cite[\href{http://dlmf.nist.gov/15.4.E6}{(15.4.6)}]{DLMF})
\begin{gather}
\label{eq:power_as_hypfun}
(1-z)^{-a} = \HypFunTO{a}{b}{b}{z},
\end{gather}
we can choose $\beta=2\lambda+3$, yielding
\begin{gather}
\label{eq:integral_cauchy}
I_+^\lambda \varphi_{2\lambda+3}(x) =
\sqrt{\frac\pi 5} \frac{\Gamma(\lambda+1)}{\Gamma\left(\lambda+\frac 32\right)} (x+1)(3-2x)^{-(\lambda+1)}.
\end{gather}
Therefore, the function given in~(\ref{eq:integral_cauchy}) is strictly positive def\/inite
on ${\mathbb S}^{2\lambda+1}$ by the remark at the start of the example.

\looseness=-1
We can follow the same line of argument applying $I_+^\mu$ to
$I_+^\lambda \varphi_{2\lambda+3}$. Again, the result is a~hypergeometric function.
Interestingly, for $\mu=\lambda-\frac 32$ the series is of the form~(\ref{eq:power_as_hypfun}) resulting in
\begin{gather*}
I_+^{\lambda-\frac 32}\big( I_+^\lambda \varphi_{2\lambda+3} \big)(x)
= \frac\pi 5 \frac{\Gamma\left(\lambda+\frac 12\right)}
 {\Gamma\left(\lambda+\frac 32\right)}\,
(x+1)^2 (3-2x)^{-\left(\lambda+\frac 12\right)}.
\end{gather*}
In general, for a function
\begin{gather*}
g_{m,\gamma}(x) = (x+1)^m (3-2x)^{-\gamma},
\qquad m\in{\mathbf N}_0, \qquad \gamma>0,
\end{gather*}
we obtain that
\begin{gather*}
I_+^{\gamma-m-\frac 32} g_{m,\gamma}(x) =
\sqrt{\frac\pi 5} \frac{\Gamma\left(\gamma-\frac 12\right)}
 {\Gamma(\gamma)}
(x+1)^{m+1} (3-2x)^{-\left(\gamma-\frac 12\right)}.
\end{gather*}

\subsection[Ultraspherical expansions of $f$ and $\CD^\lambda_{\pm} f$]{Ultraspherical expansions of $\boldsymbol{f}$ and $\boldsymbol{\CD^\lambda_{\pm} f}$}

The main results of this section will be Theorems~\ref{thm:half_step_D_lambda_eq_0}
and \ref{thm:half_step_D_general_case}
giving multiplier relationships between the Gegenbauer coef\/f\/icients of the (formal) series of~$f$
and those of the (formal) series of~$\CD^\lambda_{+} f$ and~$\CD^\lambda_{-}f$. These relationships will later be used to show
that the operators~$\CD^\lambda_{+}$ and $\CD^\lambda_{-}$ have the positive def\/initeness preserving
properties given in Theorem~\ref{thm:dimension_hopping_5}.

The f\/irst statement shows that the operators $\CD^{0}_{+}$, $\CD^{0}_{-}$
can be applied term by term to a~Chebyshev series to obtain the formal Legendre series of $\CD^0_{\pm} f$.
\begin{Theorem} \label{thm:half_step_D_lambda_eq_0}
Let $f \in C[-1,1]$ with $($formal$)$ Chebyshev series
\begin{gather*} 
f \sim \sum_{n=0}^{\infty} a_n T_n.
\end{gather*}
If both functions $\CD^0_{\pm} f \in C[-1,1]$, then the $($formal$)$ Legendre series
\begin{gather*}
\CD^0_{+} f \sim \sum_{n=0}^\infty b_n P_n
\end{gather*}
has coefficients
\begin{gather*}
b_{n} = (n+1) \pi a_{n+1}, \qquad n\in \N_0 ,
\end{gather*}
and the $($formal$)$ Legendre series
\begin{gather*}
\CD^{0}_{-} f \sim \sum_{n=0}^\infty c_n P_n
\end{gather*}
has coefficients
\begin{gather*}
 c_n = n \pi a_n, \qquad n\in \N_0 .
 \end{gather*}
\end{Theorem}

Similar relations between coef\/f\/icients in ultraspherical expansions hold for higher order Gegenbauer polynomials.

\begin{Theorem} \label{thm:half_step_D_general_case}
Let $\lambda>0$, and let
$f \in C[-1,1]$ have a $($formal$)$ Gegenbauer series
\begin{gather*} 
f \sim \sum_{n=0}^{\infty} a_n C^\lambda_n.
\end{gather*}
If both functions $\CD^\lambda_{\pm} f \in C[-1,1]$, then the $($formal$)$ Gegenbauer series
\begin{gather*}
\CD^\lambda_{+} f \sim \sum_{n=0}^\infty b_n C_n^{\lambda+\half}
\end{gather*}
has coefficients
\begin{gather}
\label{eq:multiplier_b_n}
b_n= \frac{\Gamma(\lambda+\half) \sqrt{\pi}}{\Gamma(\lambda)}\,
 \frac{2(n+2\lambda+1)}{n+\lambda+1} a_{n+1}, \qquad n\in \N_0,
\end{gather}
and the $($formal$)$ Gegenbauer series
\begin{gather*}
\CD^\lambda_{-}f \sim \sum_{n=0}^\infty c_n C^{\lambda+\half}_n,
\end{gather*}
has coefficients
\begin{gather}
\label{eq:multiplier_c_n}
c_n = \frac{\Gamma(\lambda+\half) \sqrt{\pi}}{\Gamma(\lambda)} \frac{2n}{n+\lambda} a_n, \qquad n\in \N_0 .
 \end{gather}
\end{Theorem}

\begin{Remark} Theorem~\ref{thm:half_step_D_lambda_eq_0} is the limiting
case of Theorem~\ref{thm:half_step_D_general_case} under the limit
\begin{gather}
\label{eq:limit_lambda_to_zero}
C_n^0(x) =
\lim_{\lambda\to 0^+} \frac 1\lambda
C_n^\lambda(x),
\qquad n> 0,
\qquad\text{and}\qquad C_0^0(x)=T_0(x) = 1.
\end{gather}
Furthermore, we have the special cases
\begin{gather*}
C_n^0(x)=\frac 2n T_n(x),\qquad n>0,
\end{gather*}
while
\begin{gather*} C_n^{\frac 12}(x)=P_n(x) \qquad
\text{and} \qquad C_n^1(x)=U_n(x), \qquad n\geq 0.
\end{gather*}
\end{Remark}

Before proving the theorems we state the following
technical lemma.

\begin{Lemma}
For $\lambda> 0$, $n\in \N_0$ and $x\in[-1,1]$,
\begin{gather}
\frac{d}{dx}\big\{ (1+x) C_n^\lambda(x) \big\}
 =
(n+1) C_n^\lambda(x) + 2\sum_{k=0}^{n-1} (k+\lambda) C_k^\lambda(x), \label{eq:1+x_formula} \\
\frac{d}{dx}\big\{ (1-x) C_n^\lambda(x) \big\}
 =
-(n+1) C_n^\lambda (x) + 2\sum_{k=0}^{n-1} (-1)^{k+n+1} (k+\lambda) C_k^\lambda(x).
\label{eq:1-x_formula}
\end{gather}
\end{Lemma}

\begin{proof}
Formula (\ref{eq:1-x_formula}) can be obtained from equation (\ref{eq:1+x_formula}) by using the ref\/lection formula
for Gegenbauer polynomials
\begin{gather}
\label{eq:reflection_formula}
C^\lambda_n(-x) = (-1)^n C^\lambda_n(x).
\end{gather}
and the change of variables $y=-x$.

For the proof of formula (\ref{eq:1+x_formula}) we will use two recurrences involving derivatives of Gegenbauer polynomials which can be found, for example in~\cite[10.9(35)]{Er53ii}. For notational convenience
we use the (non-standard) notation $D_n^\lambda(x)=\frac d{dx}C_n^\lambda(x)$ within
 the proof of the lemma. Then,
\begin{gather}
\label{formula:x_Cn_prime}
n C^\lambda_n(x) = x D^\lambda_n(x) - D^\lambda_{n-1}(x), \qquad\mbox{ and} \\
\label{formula:weight_Cn_prime}
\big(1-x^2\big)D^\lambda_n(x) =
\big(1-x^2\big) 2\lambda C_{n-1}^{\lambda+1}(x) =
(n+2\lambda -1) C^\lambda_{n-1}(x) - nx C^\lambda_n(x).
\end{gather}
Turn now to an induction proof of formula (\ref{eq:1+x_formula}). The statement is clearly true for
$C^\lambda_0(x)=1$, adopting the convention that the sum then is empty.

Now assume that $n\in \N$ and that the f\/irst statement is true for $n-1$. Consider
\begin{gather}
\label{start_of_induction}
\frac d{dx}\big\{ (1+x)C^\lambda_n(x)\big\} =
C_n^\lambda(x)+(1+x)D_n^\lambda(x).
\end{gather}
Using (\ref{formula:weight_Cn_prime}) and then (\ref{formula:x_Cn_prime}) we obtain that
\begin{gather*}
(1+x)D_n^\lambda(x)
 =
\big(1-x^2\big) D_n^\lambda(x) + (1+x)x D_n^\lambda(x) \\
\hphantom{(1+x)D_n^\lambda(x)}{} =
(n+2\lambda-1) C_{n-1}^\lambda(x) -nxC_n^\lambda(x)
+ (1+x) \big( n C_n^\lambda(x) + D_{n-1}^\lambda(x) \big) \\
\hphantom{(1+x)D_n^\lambda(x)}{} =
(n+2\lambda-1) C_{n-1}^\lambda(x) + n C_n^\lambda(x) + (1+x) D_{n-1}^\lambda(x).
\end{gather*}
Therefore, applying (\ref{start_of_induction}),
\begin{gather*}
\frac d{dx}\big\{ (1+x)C_n^\lambda(x)\big\}
 =
(n+2\lambda -1) C_{n-1}^\lambda(x)+(n+1) C_n^\lambda(x) + (1+x)D_{n-1}^\lambda(x) \\
\hphantom{\frac d{dx}\big\{ (1+x)C_n^\lambda(x)\big\}}{} =
(n+2\lambda-1) C_{n-1}^\lambda(x)+(n+1)C_n^\lambda(x)\\
\hphantom{\frac d{dx}\big\{ (1+x)C_n^\lambda(x)\big\}=}{}
+\frac d{dx}\big\{ (1+x) C_{n-1}^\lambda(x)\big\} - C_{n-1}^\lambda(x) \\
\hphantom{\frac d{dx}\big\{ (1+x)C_n^\lambda(x)\big\}}{}=
(n+2\lambda-2) C_{n-1}^\lambda(x) + (n+1) C_n^\lambda(x)
+ \frac d{dx}\big\{ (1+x) C_{n-1}^\lambda(x)\big\}.
\end{gather*}
Using the induction hypothesis gives
\begin{gather*}
\frac d{dx}\big\{ (1+x)C_n^\lambda(x)\big\}
 =
(n+1) C_n^\lambda(x) + (n+2\lambda-2) C_{n-1}^\lambda(x) \\
\hphantom{\frac d{dx}\big\{ (1+x)C_n^\lambda(x)\big\}=}{}
+ n C_{n-1}^\lambda(x)
+ 2\sum_{k=0}^{n-2} (k+\lambda) C_k^\lambda(x) \\
\hphantom{\frac d{dx}\big\{ (1+x)C_n^\lambda(x)\big\}}{}
 =
(n+1) C_n^\lambda(x) + 2\big[ (n-1)+\lambda \big] C_{n-1}^\lambda(x)
+ 2\sum_{k=0}^{n-2} (k+\lambda) C_k^\lambda(x),
\end{gather*}
which completes the proof.
\end{proof}

\begin{proof}[Proof of Theorem~\ref{thm:half_step_D_lambda_eq_0}]
Note that the continuity of the two functions $\CD^0_{\pm}f$ implies that of the functions $D^0_{\pm}f$.
Then, proceeding by integration by parts,
the coef\/f\/icient $b_n$ of $P_n$ in the formal Legendre expansion of $D_+^0f$ is $(2n+1)/2$ times
\begin{gather*}
H_+ =\int_{-1}^1 P_n(x) \left( D_{+}^{0} f \right)(x) dx \\
\hphantom{H_+}{} = \int_{-1}^1 P_n(x)(1+x) \frac{d}{dx} \left( \int_{-1}^x (x-t)^{-1/2} (1+t)^{-1/2} f(t) dt \right) dx\\
\hphantom{H_+}{}= \left[ P_n(x)(1+x) \int_{-1}^x (x-t)^{-1/2} (1+t)^{-1/2} f(t) dt \right]_{-1}^1 \\
\hphantom{H_+=}{} -\int_{-1}^1 \int_{-1}^x
(x-t)^{-1/2} (1+t)^{-1/2} f(t) dt
\frac{d}{dx} \big\{ P_n(x) (1+x) \big\} dx.
\end{gather*}
In view of the formula \cite[\href{http://dlmf.nist.gov/5.12.E1}{(5.12.1)}]{DLMF},
 \begin{gather*}
 \int_{-1}^x (x-t)^{-1/2}(1+t)^{-1/2} dt = \pi,
 \end{gather*}
 for all $ -1<x \leq1$, the limit as $x$ tends to $-1$ of the quantity within the square brackets vanishes. Hence,
 \begin{gather*}
H_+ =2P_n(1) \int_{-1}^1 \big(1-t^2\big)^{-1/2} f(t) dt \\
\hphantom{H_+ =}{} - \int_{-1}^1 \int_{t}^1 (x-t)^{-1/2} \frac{d}{dx} \{ P_n(x)(1+x) \} dx
 (1+t)^{-1/2} f(t) dt\\
\hphantom{H_+}{} = 2 \pi a_0 - \int_{-1}^1 \int_{t}^1 (x-t)^{-1/2}\left[ (n+1)P_n(x) +\sum_{k=0}^{n-1}
 (2k+1) P_k(x) \right] dx (1+t)^{-1/2} f(t) dt,
 \end{gather*}
 where the last step follows from an application of formula~\eqref{eq:1+x_formula}.

Noting the relationship
(cf.\ \cite[\href{http://dlmf.nist.gov/18.17.E46}{(18.17.46)}]{DLMF})
\begin{gather}
\label{eq:formula_Abramowitz}
\int_{t}^1 (x-t)^{-1/2} P_k (x) dx = \frac{1}{(k+\frac{1}{2})} \frac{1}{\sqrt{1-t}}
\big[T_k(t) -T_{k+1}(t) \big],
\end{gather}
after some straightforward calculation,
the double integral above turns into the form
\begin{gather*}
H_+ =
\frac{2n}{2n+1} \int_{-1}^1 \big(1-t^2\big)^{-\frac 12} T_n(t)
f(t) dt +
\frac{2(n+1)}{2n+1}
\int_{-1}^1 \big(1-t^2\big)^{-\frac 12} T_{n+1}(t)
f(t) dt.
\end{gather*}
Therefore,
\begin{gather*}
H_+ = \frac{n}{2n+1} \pi a_n +
\frac{n+1}{2n+1} \pi a_{n+1}.
\end{gather*}
Analogously, we def\/ine
\begin{gather*}
H_- = \int_{-1}^1 P_n(x) \left( D_{-}^{0} f\right)(x) dx.
\end{gather*}
A similar integration by parts argument, but now using the formula \eqref{eq:1-x_formula}, and the relationship
(cf.\ \cite[\href{http://dlmf.nist.gov/18.17.E45}{(18.17.45)}]{DLMF})
\begin{gather}
\label{eq:formula_Abramowitz2}
\int_{-1}^t (t-x)^{-1/2}P_k(x) dx = \frac{1}{(k+\frac{1}{2})} \frac{1}{\sqrt{1+t}}\big[T_k(t)+T_{k+1}(t) \big],
\end{gather}
shows that
\begin{gather*}
H_- = -\frac{n}{2n+1} \pi a_n + \frac{n+1}{2n+1} \pi a_{n+1}.
\end{gather*}
Since $b_n=\frac{2n+1}{2} (H_+ + H_-)$
we f\/inally obtain
\begin{gather*}
b_{n}= (n+1) \pi a_{n+1}, \qquad n \in \N_0 .
\end{gather*}
This completes the proof of the part of the theorem concerning $\CD^{0}_{+}f$.
The proof of the part of the theorem concerning $\CD^{0}_{-}f $ is similar and will be omitted.
\end{proof}

The proof of Theorem~\ref{thm:half_step_D_general_case}
relies on a kind of fractional integration by parts.
Before going into details, we will state some technical
lemmas.

\begin{Lemma}\label{lem:deriv_involving_Q} For $\lambda \geq 1/2$, $n \in \N_0$ and $x\in [-1,1]$,
\begin{gather} \label{eq:deriv_involving_Q}
\frac{d}{dx} \big\{ (1+x)^{\lambda+1} (1-x)^\lambda C^{\lambda +\half}_n (x)\big\} =
(1+x)\big(1-x^2\big)^{\lambda-1} Q_{n+1}(x),
\end{gather}
where
\begin{gather*}
Q_{n+1}(x) = (1-x)C_n^{\lambda+\frac 12}(x)
 -(n+1)\frac{2\lambda+n}{2\lambda-1}
 C_{n+1}^{\lambda-\frac 12}(x).
\end{gather*}
\end{Lemma}

\begin{proof}
Note the formula (see \cite[(22.13.2)]{Abramowitz} or
 \cite[\href{http://dlmf.nist.gov/18.17.E1}{(18.17.1)}]{DLMF} for the general Jacobi case)
\begin{gather*}
n \left(1 +\frac{n}{2\lambda} \right)
\int_0^x \big(1-t^2\big)^{\lambda-\half} C^\lambda_n (t) dt = C^{\lambda+1}_{n-1}(0) - \big(1-x^2\big)^{\lambda+\half} C^{\lambda+1}_{n-1}(x),
\end{gather*}
which implies
\begin{gather*}
\frac{d}{dx} \big\{ \big(1-x^2\big)^\lambda C^{\lambda+\half}_n(x) \big\} =-(n+1)\left(1 +\frac{n+1}{2(\lambda-\half)}\right)
 \big(1-x^2\big)^{\lambda-1} C^{\lambda-\half}_{n+1}(x), \qquad \lambda > 1/2.
 \end{gather*}

Employing the relationship above, computing the derivative on the left hand
side of \eqref{eq:deriv_involving_Q} yields
\begin{gather*}
\frac{d}{dx} \big\{ \big(1-x^2\big)^\lambda
C^{\lambda+\half}_n (x)(1+x) \big\}
 = (1-x^2)^\lambda C^{\lambda+\half}_n(x)
+(1+x)\frac{d}{dx}\big\{ \big(1-x^2\big)^\lambda C^{\lambda+\half}_n(x)\big\} \\
 =
\big(1-x^2\big)^\lambda C^{\lambda+\half}_n(x)
- (1+x) (n+1) \left(1+\frac{n+1}{2\lambda-1}\right) \big(1-x^2\big)^{\lambda-1} C^{\lambda-\half}_{n+1}(x)\\
 =
(1+x)\big(1-x^2\big)^{\lambda-1}\left\{
(1-x)C^{\lambda+\half}_n(x)
 - (n+1) \left(1 + \frac{n+1}{2\lambda-1} \right)C^{\lambda-\half}_{n+1} (x) \right\} ,
\qquad \lambda > 1/2.
\end{gather*}
Setting
\begin{gather*}
Q_{n+1}(x) = (1-x) C^{\lambda+\half}_n(x) -
 (n+1)\left(1 + \frac{n+1}{2\lambda-1}\right) C^{\lambda-\half}_{n+1}(x) ,
 \end{gather*}
 completes the proof when $\lambda > 1/2$. The limit relation (\ref{eq:limit_lambda_to_zero}) implies the result for $\lambda=1/2$.
\end{proof}

\begin{Lemma}
\label{lem:Gegenbauer_identity}
Let $\lambda>1$, $n\in{\mathbb N}$ and
$x\in[-1,1]$. Then
\begin{gather} \label{eq:geg_identity_1}
(n+2\lambda-1) C_{n+1}^{\lambda-1}(x) -
(n+2)C_{n+2}^{\lambda-1}(x) =
(2\lambda-2)(1-x)\big[ C_{n+1}^\lambda(x)+C_n^\lambda(x)\big].
\end{gather}
\end{Lemma}

\begin{proof}
Since $\frac d{dx} C_n^\lambda(x)=2\lambda
C_{n-1}^{\lambda+1}(x)$, we have that
\begin{gather*}
(2\lambda-2)\big[ C_{n+1}^\lambda(x) + C_n^\lambda(x)\big]
 =
\frac d{dx} C_{n+2}^{\lambda-1}(x) +
\frac d{dx} C_{n+1}^{\lambda-1}(x).
\end{gather*}
Using (cf.~\cite[10.9(25), (15)]{Er53ii})
\begin{gather*}
\frac d{dx} C_{n+1}^\lambda(x) =
x \frac d{dx} C_n^\lambda(x) + (2\lambda+n) C_n^\lambda(x)
\end{gather*}
and
\begin{gather*}
\big(1-x^2\big) \frac d{dx} C_n^\lambda(x) =
(n+2\lambda)x C_n^\lambda(x) - (n+1) C_{n+1}^\lambda(x)
\end{gather*}
we can proceed, obtaining
\begin{gather*}
(1-x)(2\lambda-2)\big[ C_{n+1}^\lambda(x) + C_n^\lambda(x)\big] =
(1-x)(2\lambda+n-1)C_{n+1}^{\lambda-1}(x) \\
\hphantom{(1-x)(2\lambda-2)\big[ C_{n+1}^\lambda(x) + C_n^\lambda(x)\big] =}{} +
(n+2\lambda-1)x C_{n+1}^{\lambda-1}(x) - (n+2) C_{n+1}^{\lambda-1}(x),
\end{gather*}
from which the statement follows.
\end{proof}

The following proposition is the limit case of~(\ref{eq:geg_identity_1}) taking the limit $\lambda\to 1^+$
after multiplying either side with $1/(\lambda-1)$.

\begin{Proposition}
Let $n\in\N_0$ and $x\in[-1,1]$. Then
\begin{gather*}
T_{n+1}(x) - T_{n+2}(x) =
(1-x) \big[ U_{n+1}(x)+U_n(x) \big].
\end{gather*}
\end{Proposition}

\begin{proof} From \cite[10.11(3)]{Er53ii} we have
that
\begin{gather*}
T_{n+1}(x) = U_{n+1}(x) - x U_n(x).
\end{gather*}
Furthermore, \cite[10.11(37)]{Er53ii} yields
\begin{gather*}
x U_{n+1}(x) = U_{n+2}(x) - T_{n+2}(x), \qquad
2 T_{n+2}(x) = U_{n+2}(x) - U_n(x).
\end{gather*}
The claim follows from using these relations to rewrite
the right hand side in terms of Chebyshev polynomials of the
f\/irst kind.
\end{proof}

The following lemma states the analogues of
equations~(\ref{eq:formula_Abramowitz})
and (\ref{eq:formula_Abramowitz2}) for general
$\lambda>0$. Note that the f\/irst integral given in the lemma
is a special case of \emph{Bateman's integral}
(see \cite[\href{http://dlmf.nist.gov/18.17.E9}{(18.17.9)}]{DLMF}).
Using this, the statement follows from
\cite[\href{http://dlmf.nist.gov/18.9.E4}{(18.9.4)}]{DLMF}. Rather
than using Bateman's integral for Jacobi polynomials, we provide a proof
staying within the family of Gegenbauer polynomials, only.

\begin{Lemma}
\label{lem:frac_integral_Gegenbauer}
Let $\lambda>0$, $n\in{\mathbb N}_0$ and $x\in[-1,1]$. Then
\begin{gather*}
\frac 1{\Gamma(\lambda)}
\int_x^1 (1-t)^\lambda (t-x)^{-\frac 12}
C_n^{\lambda+\frac 12}(t) dt \\
\qquad{} =
\frac{\sqrt\pi}
 {2\Gamma\left(\lambda+\frac 12\right)
 \left(n+\lambda+\frac 12\right)}
 (1-x)^{\lambda-\frac 12}
\big[ (n+2\lambda)C_n^\lambda(x)
- (n+1)C_{n+1}^\lambda(x)\big],
\end{gather*}
and
\begin{gather*}
\frac 1{\Gamma(\lambda)}
\int_{-1}^x (1+t)^\lambda (x-t)^{-\frac 12}
C_n^{\lambda+\frac 12}(t) dt \\
\qquad{} = \frac{\sqrt\pi}
 {2\Gamma\left(\lambda+\frac 12\right)
 \left(n+\lambda+\frac 12\right)}
 (1+x)^{\lambda-\frac 12}
\big[ (n+2\lambda)C_n^\lambda(x)
+ (n+1)C_{n+1}^\lambda(x)\big].
\end{gather*}
\end{Lemma}

\begin{proof}
From the def\/inition we have that
\begin{gather*}
\int_x^1 (1-t)^\lambda (t-x)^{-\frac 12}
C_n^{\lambda+\frac 12}(t) dt =
(1-x)^{\lambda-\frac 12} I_-^\lambda C_n^{\lambda+\frac 12}(x),
\end{gather*}
and $2I_-^\lambda=\CI_+^\lambda - \CI_-^\lambda$. We can
therefore use Theorem~\ref{thm:action_on_polynomials} below to
compute
\begin{gather*}
2I_-^\lambda C_n^{\lambda+\frac 12}(x) =
\frac{\sqrt\pi\Gamma(\lambda)}
 {\Gamma\left(\lambda+\frac 12\right)
 \left(n+\lambda+\frac 12\right)}
\big[(n+2\lambda)C_n^\lambda(x) -
 (n+1) C_{n+1}^\lambda(x) \big],
\end{gather*}
from which the f\/irst statement follows.
The second integral follows from the f\/irst by a change of variables $\tau=-t$ and using the ref\/lection formula~(\ref{eq:reflection_formula}).
\end{proof}

\begin{Remark} Taking the limit (\ref{eq:limit_lambda_to_zero})
readily leads to~(\ref{eq:formula_Abramowitz}) and~(\ref{eq:formula_Abramowitz2}).
\end{Remark}

\begin{Corollary}
\label{cor:frac_integral_Gegenbauer}
Let $\lambda\geq 1$, $n\in{\mathbb N}_0$ and $x\in[-1,1]$. Then
\begin{gather*}
\frac 1{\Gamma(\lambda)}
\int_x^1 (1-t)^{\lambda-1} (t-x)^{-\frac 12}
C_{n+1}^{\lambda-\frac 12}(t) dt\\
\qquad{} =
\frac{\sqrt\pi}
 {\Gamma\left(\lambda-\frac 12\right)
 \left(n+\lambda+\frac 12\right)}
 (1-x)^{\lambda-\frac 12}
\big[ C_{n+1}^\lambda(x)
+ C_n^\lambda(x)\big],
\end{gather*}
and
\begin{gather*}
\frac 1{\Gamma(\lambda)}
\int_{-1}^x (1+t)^{\lambda-1} (x-t)^{-\frac 12}
C_{n+1}^{\lambda-\frac 12}(t) dt\\
\qquad{}
=
\frac{\sqrt\pi}
 {\Gamma\left(\lambda-\frac 12\right)
 \left(n+\lambda+\frac 12\right)}
 (1+x)^{\lambda-\frac 12}
\big[ C_{n+1}^\lambda(x)
- C_n^\lambda(x)\big].
\end{gather*}
\end{Corollary}

\begin{proof} The result follows from
setting $\lambda-1$ and $n+1$ instead of
$\lambda$ and $n$, respectively in
Lemma~\ref{lem:frac_integral_Gegenbauer}
and applying
Lemma~\ref{lem:Gegenbauer_identity}.
\end{proof}

\begin{Remark} Formul{\ae} (\ref{eq:formula_Abramowitz})
and~(\ref{eq:formula_Abramowitz2}) are the special cases
for $\lambda=1$.
\end{Remark}

\begin{proof}[Proof of Theorem~\ref{thm:half_step_D_general_case}]
The Fourier--Gegenbauer coef\/f\/icients of the function $\CD_+^\lambda f$
are given by
\begin{gather*}
b_n =
\frac 1{h_n^{\lambda+\frac 12}}
\int_{-1}^1 \CD_+^\lambda f(t) C_n^{\lambda+\frac 12}(t)
\big(1-t^2\big)^\lambda dt,
\end{gather*}
where
\begin{gather*}
h_n^{\lambda+\frac 12} =
\frac{\pi\Gamma(2\lambda+n+1)}
 {2^{2\lambda}n! \left(\lambda+n+\frac 12\right)
 \Gamma^2\left(\lambda+\frac 12\right)}.
\end{gather*}
Therefore,
\begin{gather*}
h_n^{\lambda+\frac 12} b_n =
\int_{-1}^1 D_+^\lambda f(t) C_n^{\lambda+\frac 12}(t)
\big(1-t^2\big)^\lambda dt
+
\int_{-1}^1 D_-^\lambda f(t) C_n^{\lambda+\frac 12}(t)
\big(1-t^2\big)^\lambda dt.
\end{gather*}
Let us denote the f\/irst integral by $H_+$ and the second
by~$H_-$. From integration by parts it follows that
\begin{gather*}
H_+ =
\int_{-1}^1 (1+x) \frac d{dx}
\left\{ (1+x)^{-\lambda}
\int_{-1}^x (x-t)^{-\frac 12}(1+t)^{\lambda-\frac 12}
f(t) dt \right\} C_n^{\lambda+\frac 12}(x)
\big(1-x^2\big)^\lambda dx \\
\hphantom{H_+ }{} =
\left[ (1+x)^{-\lambda}
\int_{-1}^x (x-t)^{-\frac 12}(1+t)^{\lambda-\frac 12}
f(t) dt (1+x)\big(1-x^2\big)^\lambda C_n^{\lambda+\frac 12}(x)
\right]_{-1}^1 \\
\hphantom{H_+ = }{} -
\int_{-1}^1 (1+x)^{-\lambda}
\int_{-1}^x (x-t)^{-\frac 12}(1+t)^{\lambda-\frac 12}
f(t) dt
\frac d{dx}\big\{ (1+x)\big(1-x^2\big)^\lambda
C_n^{\lambda+\frac 12}(x)\big\} dx.
\end{gather*}
The constant term vanishes.
Applying Lemma~\ref{lem:deriv_involving_Q} we can
decompose $H_+$ into the sum of the two integrals
\begin{gather*}
I_1 =
-\int_{-1}^1 (1+x)^{-\lambda}
\int_{-1}^x (x-t)^{-\frac 12} (1+t)^{\lambda-\frac 12}
f(t) dt
\big(1-x^2\big)^\lambda C_n^{\lambda+\frac 12}(x) dx \\
\hphantom{I_1}{} =
-\int_{-1}^1 \int_t^1 (1-x)^\lambda (x-t)^{-\frac 12}
C_n^{\lambda+\frac 12}(x) dx
(1+t)^{\lambda-\frac 12} f(t) dt
\end{gather*}
and
\begin{gather*}
I_2 = (n+1)\frac{2\lambda+n}{2\lambda-1}
\int_{-1}^1 (1+x)^{-\lambda+1}
\int_{-1}^x (x-t)^{-\frac 12} (1+t)^{\lambda-\frac 12}
f(t) dt
\big(1-x^2\big)^{\lambda-1} C_{n+1}^{\lambda-\frac 12}(x) dx \\
\hphantom{I_2}{} =
(n+1)\frac{2\lambda+n}{2\lambda-1}
\int_{-1}^1 \int_t^1 (1-x)^{\lambda-1} (x-t)^{-\frac 12}
C_{n+1}^{\lambda-\frac 12}(x) dx
(1+t)^{\lambda-\frac 12} f(t) dt.
\end{gather*}
In an analogous way, we decompose~$H_-$ into a sum of the
integrals
\begin{gather*}
I_3 =
\int_{-1}^1 \int_{-1}^t (1+x)^\lambda (t-x)^{-\frac 12}
C_n^{\lambda+\frac 12}(x) dx
(1-t)^{\lambda-\frac 12} f(t) dt
\end{gather*}
and
\begin{gather*}
I_4 =
(n+1)\frac{2\lambda+n}{2\lambda-1}
\int_{-1}^1\int_{-1}^t
(1+x)^{\lambda-1}(t-x)^{-\frac 12}
C_{n+1}^{\lambda-\frac 12}(x) dx
(1-t)^{\lambda-\frac 12} f(t) dt.
\end{gather*}
The inner integrals in $I_1$ and $I_3$ can be computed using
Lemma~\ref{lem:frac_integral_Gegenbauer}, whereas the
corresponding formul{\ae} for $I_2$ and $I_4$ are stated in
Corollary~\ref{cor:frac_integral_Gegenbauer}.
Doing so, and using the def\/inition of the coef\/f\/icients
\begin{gather*}
a_n =
\frac 1{h_n^\lambda}
\int_{-1}^1 f(t) C_n^\lambda(t)
\big(1-t^2\big)^{\lambda-\frac 12} dt,
\qquad
\text{where}
\qquad
h_n^\lambda =
\frac{\pi \Gamma(2\lambda+n)}
 {2^{2\lambda-1}n! (\lambda+n)\Gamma^2(\lambda)},
\end{gather*}
we obtain that
\begin{gather*}
I_1 =
\frac{\pi^{\frac 32}}
 {2^{2\lambda}n!\Gamma(\lambda)
 \Gamma\left(\lambda+\frac 12\right)
 \left(n+\lambda+\frac 12\right)} \\
\hphantom{I_1 =}{}\times
\left[ \frac{\Gamma(2\lambda+n+1)}{\lambda+n+1} a_{n+1}
- (2\lambda+n)\frac{\Gamma(2\lambda+n)}{\lambda+n} a_n
\right], \\
I_2 = \frac{\pi^{\frac 32}}
 {2^{2\lambda}n!\Gamma(\lambda)
 \Gamma\left(\lambda+\frac 12\right)
 \left(n+\lambda+\frac 12\right)} \\
\hphantom{I_2 =}{}\times
\left[ (2\lambda+n)\frac{\Gamma(2\lambda+n+1)}
 {\lambda+n+1} a_{n+1}
+ (2\lambda+n)\frac{\Gamma(2\lambda+n)}{\lambda+n} a_n
\right], \\
I_3 = \frac{\pi^{\frac 32}}
 {2^{2\lambda}n!\Gamma(\lambda)
 \Gamma\left(\lambda+\frac 12\right)
 \left(n+\lambda+\frac 12\right)} \\
\hphantom{I_3 =}{}\times
\left[ (n+2\lambda) \frac{\Gamma(2\lambda+n)}
 {\lambda+n} a_n
+ \frac{\Gamma(2\lambda+n+1)}{\lambda+n+1} a_{n+1}
\right], \\
I_4 = \frac{\pi^{\frac 32}}
 {2^{2\lambda}n!\Gamma(\lambda)
 \Gamma\left(\lambda+\frac 12\right)
 \left(n+\lambda+\frac 12\right)} \\
\hphantom{I_4 =}{}\times
\left[ (2\lambda+n)\frac{\Gamma(2\lambda+n+1)}
 {\lambda+n+1} a_{n+1}
- (2\lambda+n)\frac{\Gamma(2\lambda+n)}
 {\lambda+n} a_n\right].
\end{gather*}
Therefore,
\begin{gather*}
H_+ + H_- =
\frac{\pi^{\frac 32}\Gamma(2\lambda+n+2)}
 {2^{2\lambda-1}n!\Gamma(\lambda)
 \Gamma\left(\lambda+\frac 12\right)
 \left(n+\lambda+\frac 12\right)(\lambda+n+1)}
a_{n+1},
\end{gather*}
from which it follows that
\begin{gather*}
b_n = \frac{2\sqrt\pi\Gamma\left(\lambda+\frac 12\right)}
 {\Gamma(\lambda)}
\frac{2\lambda+n+1}{\lambda+n+1} a_{n+1}.
\end{gather*}
This completes the proof of the part of the theorem concerning $\CD_+^\lambda f$. Again, the proof of the part concerning $\CD^{\lambda}_{-}f $ is similar and will therefore be omitted.
\end{proof}

\begin{proof}[Proof of Theorem~\ref{thm:dimension_hopping_5}]
From the continuity assumption we have that the Gegenbauer series of the functions $\CD^\lambda_\pm f$ are Abel summable. Since $f$ is positive def\/inite by hypothesis, Theorem~\ref{thm:half_step_D_general_case} shows that the Gegenbauer coef\/f\/icients of $\CD^\lambda_\pm f$ are non-negative. Hence the Gegenbauer series restricted to $x=1$
are series of non-negative terms. For such series of non-negative terms Abel summability
implies summability. Since $C_n^{\lambda+\frac 12}$ attains its maximum at the point 1 we can apply the Weierstra{\ss} M-test to show that the Gegenbauer series of $\CD^\lambda_\pm f$ converge uniformly on $[-1,1]$. Furthermore, the multipliers given in equations~(\ref{eq:multiplier_b_n}) and (\ref{eq:multiplier_c_n}) preserve the sign of the coef\/f\/icients. Therefore, $\CD_\pm^\lambda f\in\Lambda_{m+1}$ by
Theorem~\ref{thm:schoenberg}. The statement about strict positive-def\/initeness then follows from the same relation and the discussion on strict positive def\/initeness in
the paragraph following Theorem~\ref{thm:schoenberg}.
\end{proof}

\section[The action of $\CI^\lambda_\pm$ and $\CD^\lambda_\pm$ on Gegenbauer polynomials]{The action of $\boldsymbol{\CI^\lambda_\pm}$ and $\boldsymbol{\CD^\lambda_\pm}$ on Gegenbauer polynomials}
\label{sec:montee_and_descente_half_step}

The claim that the operators $\CI^\lambda_\pm$ and $\CD^\lambda_\pm$ map Gegenbauer polynomials onto Gegenbauer polynomials with a changed parameter is based upon contiguous relations for
 hypergeometric functions. We will thus f\/irst state a proposition showing that the images of $C_n^{\lambda+\frac 12}$ under the operators $I_\pm^\lambda$ are hypergeometric polynomials.

\begin{Proposition} \label{prop:Ioperators}
Let $\lambda\geq 0$, $n\in\N_0$ and $x\in[-1,1]$. Then
\begin{gather}I^\lambda_{+} C_n^{\lambda + \frac 12} (x) =
c_{n,\lambda}
\frac{\lambda + \frac 12}{\lambda +n + {\frac 12}} (1+x)\, \HypFunTO{-n}{n+2\lambda +1}{\lambda+\frac 12}{\frac{1-x}{2}}, \label{eq:Iplus_on_C}
\\
I^\lambda_{-} C_n^{\lambda + \frac 12} (x) =
c_{n,\lambda}
(1-x)\, \HypFunTO{-n}{n+2\lambda +1}{\lambda+\frac 32}{\frac{1-x}{2}},
\label{eq:Iminus_on_C}
\end{gather}
where
\begin{gather*}
c_{n,\lambda} =
\frac{\sqrt{\pi} (2\lambda+1)_n}{n!} \frac{\Gamma(\lambda+1)}{\Gamma(\lambda + \frac 32)}.
\end{gather*}
\end{Proposition}

Before proving the proposition let us state a technical lemma.
\begin{Lemma} \label{lem:Ipm_on_F}
Let $\lambda>-\frac 12$, $n\in\N_0$ and $x\in[-1,1]$. Then
\begin{gather*}
\left(I^\lambda_\pm \, \HypFunTO{-n}{n+2\lambda+1}{\lambda+1}
{\frac{1\pm \Cdot}{2}}\right)(x) \\
\qquad{}=
 \sqrt{\pi} \frac{\Gamma(\lambda+1)}{\Gamma(\lambda+\frac 32)} (1\pm x)
 \, \HypFunTO{-n}{n+2\lambda+1}{\lambda+ \frac 32}{\frac{1\pm x}{2} }.
 \end{gather*}
\end{Lemma}

\begin{proof} We will prove the statement in the `$+$' case (the proof of `$-$' case being
analogous):
\begin{gather*}
\left( I^\lambda_+ \, \HypFunTO{-n}{n+2\lambda+1}{\lambda+1}{\frac{1+\Cdot}{2} } \right)(x)\\
 \qquad{} = (1+x)^{-\lambda +\frac 12} \int_{-1}^x (x-\tau)^{- \half} (1+\tau)^\lambda \,
\HypFunTO{-n}{n+2\lambda +1}{\lambda+1}{\frac{1+\tau}{2} } d \tau\\
\qquad{} = \sum_{k=0}^n
 \frac{(-n)_k (n+2\lambda +1)_k}{ (\lambda+1)_k k! 2^k}
(1+x)^{-\lambda+\half}
 \int_{-1}^x (x-\tau)^{-\half} (1+\tau)^{\lambda+k} d \tau.
 \end{gather*}
Using the beta integral~\eqref{eq:beta_integral1} the expression above reduces to
 \begin{gather*}
 \sum_{k=0}^n \frac{(-n)_k (n+2\lambda+1)_k}{(\lambda+1)_k k! 2^k}
 \frac{\Gamma(\half) \Gamma( \lambda+k+1)}{\Gamma(\lambda+k+\frac 32)} (1+x)^{k+1}\\
\qquad{} = \sqrt{\pi} (1+x) \sum_{k=0}^n \frac{(-n)_k (n+2\lambda+1)_k \Gamma(\lambda+1)}{k! 2^k}
\frac{ (1+x)^k}{(\lambda+\frac 32)_k \Gamma(\lambda+\frac 32)}\\
 \qquad{} = \sqrt{\pi}\; \frac{\Gamma(\lambda+1)}{\Gamma(\lambda+\frac 32)}
(1+x)\, \HypFunTO{-n}{n+2\lambda +1}{\lambda+\frac 32 }{\frac{1+x}{2}}.\tag*{\qed}
\end{gather*}\renewcommand{\qed}{}
\end{proof}

The following proof of Proposition~\ref{prop:Ioperators} uses an identity due to Pfaf\/f (cf.~\cite[2.3.14]{An99})
\begin{gather}
\label{eq:change_last_arg_hypergeometric}
\HypFunTO{-m}{b}{c}{z} =
\frac{(c-b)_m}{(c)_m}\,
\HypFunTO{-m}{b}{b-c-m+1}{1-z},
\qquad m\in\N_0,
\end{gather}
which will occur frequently later on.

\begin{proof}[Proof of Proposition~\ref{prop:Ioperators}] From the representation of the Gegenbauer polynomials as hypergeometric
functions \cite[\href{http://dlmf.nist.gov/18.5.E9}{(18.5.9)}]{DLMF}, and the ref\/lection formula~(\ref{eq:reflection_formula}) we have
\begin{gather*}
I^\lambda_{+} C^{\lambda+ \frac 12}_n (x) = (-1)^n
\frac{\Gamma(n+2\lambda+1)}{n! \Gamma(2\lambda+1)}
\left( I^\lambda_+ \, \HypFunTO{-n}{n+2\lambda +1}{\lambda+1}{\frac{1+\Cdot}{2}}\right)(x).
\end{gather*}
Applying Lemma \ref{lem:Ipm_on_F}
\begin{gather}
I^\lambda_{+} C^{\lambda+ \frac 12}_n (x) = (-1)^n
\frac{\Gamma(n+2\lambda+1)}{n! \Gamma(2\lambda+1)}
 \sqrt{\pi} \frac{\Gamma(\lambda+1)}{\Gamma(\lambda+\frac 32)}
(1+x)\, \HypFunTO{-n}{n+2\lambda +1}{\lambda+\frac 32 }{\frac{1+x}{2}} \nonumber \\
\hphantom{I^\lambda_{+} C^{\lambda+ \frac 12}_n (x)}{}
 = \frac{\sqrt{\pi} (-1)^n (2\lambda+1)_n}{n!} \frac{\Gamma(\lambda+1)}{\Gamma(\lambda+ \frac 32)}
(1+x)\, \HypFunTO{-n}{n+2\lambda +1}{\lambda+\frac 32 }{\frac{1+x}{2}}. \label{eq:interform_Iplus}
\end{gather}
Now \eqref{eq:change_last_arg_hypergeometric} implies
\begin{gather} \label{eq:changed_hyper}
\HypFunTO{-n}{n+2\lambda +1}{\lambda+\frac 32 }{\frac{1+x}{2}}
=
\frac{(-\lambda -n +\frac 12)_n}{(\lambda+ \frac 32)_n}\,
 \HypFunTO{-n}{n+2\lambda +1}{\lambda+\frac 12 }{\frac{1-x}{2}}.
 \end{gather}
Substituting \eqref{eq:changed_hyper} into \eqref{eq:interform_Iplus} noting that $ (a-n)_n = (-1)^n (1-a)_n$ shows
\begin{gather*}
I^\lambda_{+} C^{\lambda+ \frac 12}_n (x) =
 \frac{\sqrt{\pi} (2\lambda+1)_n}{n!} \frac{\Gamma(\lambda+1)}{\Gamma(\lambda+ \frac 32)}
\frac{(\lambda +\frac 12)_n}{(\lambda +\frac 32)_n}(1+x)\,
 \HypFunTO{-n}{n+2\lambda +1}{\lambda+\frac 12 }{\frac{1-x}{2}}\\
\hphantom{I^\lambda_{+} C^{\lambda+ \frac 12}_n (x)}{}
= \frac{\sqrt{\pi} (2\lambda+1)_n}{n!} \frac{\Gamma(\lambda+1)}{\Gamma(\lambda + \frac 32)}
\frac{\lambda + \frac 12}{\lambda +n + {\frac 12}} (1+x)\,
\HypFunTO{-n}{n+2\lambda +1}{\lambda+\frac 12}{\frac{1-x}{2}}.
 \end{gather*}

The proof of equation~\eqref{eq:Iminus_on_C} is almost identical to that part of the proof of
equation~\eqref{eq:Iplus_on_C} up to equation~\eqref{eq:interform_Iplus}.
It will therefore be omitted.
\end{proof}

\begin{Theorem}\label{thm:action_on_polynomials}
Let $\lambda> 0$ and $n\in\N_0$. Then
\begin{gather}
\CI^\lambda_{+} C^{\lambda + \frac 12}_n (x) =
\frac{ \sqrt{\pi} \Gamma(\lambda)}{\Gamma({\lambda +\frac 12})}
\frac{ n+2\lambda }{ n+\lambda+{\frac 12}} C^\lambda_n (x), \qquad x\in[-1,1], \label{eq:curlyI_plus_on_C}\\
\CI^\lambda_{-} C^{\lambda + \frac 12}_n (x) =
\frac{\sqrt{\pi} \Gamma(\lambda)}{\Gamma({\lambda +\frac 12})}
\frac{n+1}{n+\lambda+{\frac 12}} C^\lambda_{n+1} (x), \qquad
x\in[-1,1]. \label{eq:curlyI_minus_on_C}
\end{gather}
\end{Theorem}

\begin{Remark} Taking the limit of the last two relations as $\lambda\rightarrow 0^+$ (see~(\ref{eq:limit_lambda_to_zero})) gives
\begin{gather*}
 \CI^0_+ P_n (x) = \frac{2}{n+\half} T_n (x), \qquad \text{and}\qquad
\CI^0_{-} P_n(x) = \frac{2}{n+\half} T_{n+1}(x), \qquad n \in\N_0.
\end{gather*}
\end{Remark}

\begin{proof}[Proof of Theorem~\ref{thm:action_on_polynomials}]
From~\eqref{eq:Iplus_on_C} and~\eqref{eq:Iminus_on_C} we obtain that
\begin{gather*}
\CI^\lambda_{+} C^{\lambda+\half}_n (x) = \big\{ I^\lambda_{+}+I^\lambda_{-}\big\} C^{\lambda+\half}_n(x)\\
\hphantom{\CI^\lambda_{+} C^{\lambda+\half}_n (x)}{}
 = \frac{2 \sqrt{\pi} (2\lambda+1)_n}{(\lambda+n+\half) n!} \frac{\Gamma(\lambda+1)}{\Gamma(\lambda+\frac 32)}
\left\{ \left(\lambda+\half\right) \left(\frac{1+x}{2}\right)
\HypFunTO{-n}{n+2\lambda +1}{\lambda+\frac 12}{\frac{1-x}{2}}\!\right.\\
 \left.
\hphantom{\CI^\lambda_{+} C^{\lambda+\half}_n (x)=}{}
 + \left(\lambda+n+\half\right) \left(\frac{1-x}{2}\right)
\HypFunTO{-n}{n+2\lambda +1}{\lambda+\frac 32}{\frac{1-x}{2}} \right\}.
\end{gather*}
An application of the contiguous relation \cite[(15.2.25)]{Abramowitz}
then shows
\begin{gather*}
\CI^\lambda_{+} C^{\lambda+\half}_n (x)
 = \frac{2 \sqrt{\pi} (2\lambda+1)_n}{(\lambda+n+\half) n!} \frac{\Gamma(\lambda+1)}{\Gamma(\lambda+\frac 32)} \left(\lambda+\half\right)
 \HypFunTO{-n}{n+2\lambda}{\lambda+\half }{\frac{1-x}{2}} \\
\hphantom{\CI^\lambda_{+} C^{\lambda+\half}_n (x)}{}
= \frac{2 \sqrt{\pi} (2\lambda+1)_n}{(\lambda+n+\half) n!} \frac{\Gamma(\lambda+1)}{\Gamma(\lambda+\half)} \frac{n!}{(2\lambda)_n} C^\lambda_n (x)= \frac{\sqrt{\pi} (2\lambda+n) \Gamma(\lambda)}{ (\lambda+n+\half)\Gamma(\lambda+\half)} C^\lambda_n(x),
\end{gather*}
 which is equation~\eqref{eq:curlyI_plus_on_C}.

The proof of equation~\eqref{eq:curlyI_minus_on_C} is analogous except that it uses the contiguous
relation \cite[\href{http://dlmf.nist.gov/15.5.E16}{(15.5.16)}]{DLMF}. Hence, it will be omitted.
\end{proof}

For the operators $\CD^\lambda_\pm$ we obtain the following result.
\begin{Theorem}\label{thm:action_of_derivative}
Let $\lambda>0$ and $n\in\N_0$. Then
\begin{gather}
\CD^\lambda_{+} C^{\lambda}_n (x) =
\frac{ \sqrt{\pi} \Gamma(\lambda +\half )}{\Gamma({\lambda })}
\frac{ 2(n+2\lambda) }{ n+\lambda} C^{\lambda+\half}_{n-1} (x), \qquad x\in[-1,1], \label{eq:curlyD_plus_on_C}\\
\CD^\lambda_{-} C^{\lambda}_n (x) =
\frac{\sqrt{\pi} \Gamma(\lambda +\half )}{\Gamma(\lambda )}
\frac{2n}{n+\lambda} C^{\lambda+\half}_n (x),
\qquad x\in[-1,1]. \label{eq:curlyD_minus_on_C}
\end{gather}
\end{Theorem}

\begin{Remark} Taking the limit of the last two relations as $\lambda \rightarrow 0^+$ (see (\ref{eq:limit_lambda_to_zero})) gives
\begin{gather*}
\CD^0_{+} T_n(x)= n\pi P_{n-1}(x)\qquad \text{and}\qquad \CD^0_{-} T_n(x)= n \pi P_{n}(x), \qquad n \in \NN_0.
\end{gather*}
\end{Remark}

Again, we f\/irst prove a preparatory proposition.

\begin{Proposition}
Let $\lambda >0$, $n\in\N_0$ and $x\in[-1,1]$. Then
\begin{gather}D^\lambda_{+} C_n^{\lambda} (x) =
c_{n,\lambda}
\frac{n(n+2\lambda) }{n+\lambda } \frac{1+x}{2}\,
\HypFunTO{-n+1}{n+2\lambda +1}{\lambda+1}{\frac{1-x}{2}}, \label{eq:Dplus_on_C} \\
D^\lambda_{-} C_n^{\lambda } (x) =
c_{n,\lambda}
\frac{n(n+2\lambda)}{\lambda+1}
\frac{1-x}{2} \, \HypFunTO{-n+1}{n+2\lambda +1}{\lambda+2}{\frac{1-x}{2}},
\label{eq:Dminus_on_C}
\end{gather}
where
\begin{gather*}
c_{n,\lambda} =
\frac{\sqrt{\pi} (2\lambda)_n}{n!} \frac{\Gamma(\lambda+\half)}{\Gamma(\lambda + 1)}.
\end{gather*}
\end{Proposition}

\begin{proof}
Consider the f\/irst equation
 \begin{gather*}
\frac{1}{1+x} D^\lambda_+ C^\lambda_n(x) =
\frac{d}{dx} \left\{ (1+x)^{-\lambda} \int_{-1}^x (x-\tau)^{-\half} (1+\tau)^{\lambda -\half} C^\lambda_n(\tau) d\tau\right\}\\
= \frac{(-1)^n (2\lambda)_n}{n!} \frac{d}{dx} \left\{ (1+x)^{-\lambda} \int_{-1}^x (x-\tau)^{-\half} (1+\tau)^{\lambda-\half }\,
\HypFunTO{-n}{n+2\lambda}{\lambda+\half}{ \frac{1+\tau}{2} } d\tau \right\} \\
= \frac{(-1)^n (2\lambda)_n}{n!} \frac{d}{dx} \left\{ \frac{1}{1+x}I^{\lambda-{\frac 12}}_+ \left(
\HypFunTO{-n}{n+2\lambda}{\lambda+\half}{ \frac{1+\Cdot}{2} } \right)(x)
\right\}.
\end{gather*}
Applying Lemma~\ref{lem:Ipm_on_F}
\begin{gather*}
\frac{1}{1+x} D^\lambda_+ C^\lambda_n(x)
 = \frac{(-1)^n (2\lambda)_n}{n!}
 \sqrt{\pi} \frac{\Gamma(\lambda+\half)}{\Gamma(\lambda+1)} \frac{d}{dx}\,
\HypFunTO{-n}{n+2\lambda}{\lambda+1}{ \frac{1+x}{2} }.
\end{gather*}
The formula (cf.~\cite[\href{http://dlmf.nist.gov/15.5.E1}{(15.2.1)}]{DLMF})
\begin{gather*}
\frac{d}{dx} \HypFunTO{a}{b}{c}{ x }
= \frac{(a)_1(b)_1}{(c)_1} \, \HypFunTO{a+1}{b+1}{c+1}{ x }
\end{gather*}
then implies
\begin{gather}
D^\lambda_+ C^\lambda_n(x)
 = \frac{(-1)^n (2\lambda)_n}{n!}
 \sqrt{\pi} \frac{\Gamma(\lambda+\half)}{\Gamma(\lambda+1)}\frac{1+x}{2} \frac{(-n)(n+2\lambda)}{(\lambda+1)}\nonumber\\
 \hphantom{D^\lambda_+ C^\lambda_n(x)=}{}\times
\HypFunTO{-n+1}{n+2\lambda+1}{\lambda+2}{ \frac{1+x}{2} } \nonumber\\
\hphantom{D^\lambda_+ C^\lambda_n(x)}{}
= \frac{\sqrt{\pi} (2\lambda)_n \Gamma(\lambda+\half)}{n! \Gamma(\lambda+1)}
(-1)^{n+1}\frac{1+x}{2} \frac{(n)(n+2\lambda)}{(\lambda+1)}\nonumber\\
 \hphantom{D^\lambda_+ C^\lambda_n(x)=}{}\times
\HypFunTO{-n+1}{n+2\lambda+1}{\lambda+2}{ \frac{1+x}{2} }. \label{eq:Dlambda_plus_partway_on_F}
\end{gather}
Now from equation~\eqref{eq:change_last_arg_hypergeometric}
\begin{gather*}
\HypFunTO{-n+1}{n+2\lambda+1}{\lambda+2}{ \frac{1+x}{2} }
=\frac{(-1)^{n-1} (1+\lambda)}{n+\lambda}\,
\HypFunTO{-n+1}{n+2\lambda+1}{\lambda+1}{ \frac{1-x}{2} }.
\end{gather*}
Substituting into equation~\eqref{eq:Dlambda_plus_partway_on_F}
gives equation~\eqref{eq:Dplus_on_C}.

The proof of equation~\eqref{eq:Dminus_on_C} is almost identical to that of equation~\eqref{eq:Dplus_on_C}.
It will therefore be omitted.
\end{proof}

\begin{proof}[Proof of Theorem~\ref{thm:action_of_derivative}]
From \eqref{eq:Dplus_on_C} and \eqref{eq:Dminus_on_C}
\begin{gather*}
\CD^\lambda_+ C_n^\lambda(x) =
\frac{\sqrt{\pi} (2\lambda)_n}{n!} \frac{\Gamma(\lambda+\half)}{\Gamma(\lambda+1)} \frac{n(n+2\lambda)}{(n+\lambda)(\lambda+1)} \\
\hphantom{\CD^\lambda_+ C_n^\lambda(x) =}{}
 \times
\left\{
(\lambda+1)\left(\frac{1+x}{2}\right)
\HypFunTO{-n+1}{n+2\lambda+1}{\lambda+1}{ \frac{1-x}{2} }\right.\\
 \left.
\hphantom{\CD^\lambda_+ C_n^\lambda(x) =}{}
 +(n+\lambda) \left( \frac{1-x}{2}\right)
\HypFunTO{-n+1}{n+2\lambda+1}{\lambda+2}{ \frac{1-x}{2} }
 \right\}.
 \end{gather*}
An application of the contiguous relation~\cite[(15.2.25)]{Abramowitz} then shows
\begin{gather*}
\CD^\lambda_+ C_n^\lambda(x) =
\frac{\sqrt{\pi} (2\lambda)_n}{n!} \frac{\Gamma(\lambda+\half)}{\Gamma(\lambda+1)} \frac{n(n+2\lambda)}{(n+\lambda)}\, \HypFunTO{-n+1}{n+2\lambda}{\lambda+1}{ \frac{1-x}{2} } ,
\end{gather*}
and rewriting the hypergeometric function on the right as a Gegenbauer polynomial
gives~\eqref{eq:curlyD_plus_on_C}.

The proof of equation~\eqref{eq:curlyD_minus_on_C} is analogous except that it uses contiguous
relation \cite[\href{http://dlmf.nist.gov/15/5/E16}{(15.5.16)}]{DLMF}. Hence, it will be omitted.
\end{proof}

\subsection*{Acknowledgements}

The authors thank the editor and the referees for their helpful suggestions.
WzC acknowledges support from a University of Canterbury Visiting
Erskine Fellowship. RKB is grateful for the hospitality provided by the
Helmholtz Zentrum M\"{u}nchen.

\pdfbookmark[1]{References}{ref}
\LastPageEnding

\end{document}